\documentclass{article}
\usepackage[a4paper,left=2cm,right=2cm,top=2cm,bottom=2cm]{geometry}
\usepackage[cp1250]{inputenc}
\usepackage{amsfonts,amsmath,upref,amssymb,amsthm,amsxtra,amscd,enumerate}
\DeclareMathOperator{\diver}{div}

\newcommand{\dx}{\,\mathrm{d}x}

\newcommand{\dt}{\,\mathrm{d}t}

\newcommand{\bb}{{\bf b}}
\newcommand{\vb}{\bb}
\newcommand{\be}{{\bf e}}
\newcommand{\ve}{\be}
\newcommand{\bw}{{\bf w}}
\newcommand{\vw}{\bw}
\newcommand{\vf}{{\bf f}}
\newcommand{\bF}{{\bf F}}

\newcommand{\bn}{{\bf n}}
\newcommand{\vn}{\bn}
\newcommand{\bu}{{\bf u}}
\newcommand{\vu}{\bu}
\newcommand{\bv}{{\bf v}}
\newcommand{\vv}{\bv}

\newcommand{\pat}{\partial_t}

\newtheorem{theorem}{Theorem}[section]
\newtheorem{lemma}[theorem]{Lemma}
\newtheorem{definition}[theorem]{Definition}

\numberwithin{equation}{section}

\title{Compressible fluid inside a linear oscillator}
\date{}
\author{V\' aclav M\' acha\footnote{Institute of Mathematics of the Academy of Sciences of the Czech republic,
E-mail: macha@math.cas.cz}}
\begin{document}
\maketitle

\begin{abstract}
We use Feireisl-Lions theory to deduce the existence of weak solutions to a system describing the dynamics of a linear oscillator containing a Newtonian compressible fluid. The appropriate Navier-Stokes equation is considered on a domain whose movement has one degree of freedom. The equation is paired with the Newton law and we assume a no-slip boundary condition. 
\end{abstract}

\section{Introduction}

It is known that the presence of a fluid inside a freely moving body has a tremendous effect on the stability of the system. This has been investigated several times, most recently by Disser, Galdi, Mazzone and Zunino \cite{DiGaMaZu} for incompressible fluid inside a rotating body, Galdi and Mazzone \cite{GaMa} for incompressible fluid inside a pendulum (see also \cite{GaMaMo}). In both cases there appear some non-potential forces which cause the fluid to move unless the motion of the whole system is stabilized. The energy of that flow then dissipate because of the viscosity of the fluid.

The same effect also appears in the case of compressible fluid inside a body -- we refer to \cite{GaMaNe} and \cite{GaMaNeSh}. In \cite{GaMaNeSh} authors tackle a system consisting of a pendulum containing a compressible fluid. The numerical simulations provided there showed that the compressible fluids provide better damping than the incompressible ones. This may be explained by a simple idea -- there appears a flow even if the oscillator is linear (there is no swing). Consequently, the mechanical energy of the whole system dissipates because of the viscosity term. For more see the next subsection.

The main aim of the paper is to prove the existence of solution to such system, i.e., a linear oscillator consisting of a spring and a container filled by a compressible fluid.  A container $\Omega\subset \mathbb R^{3}$,  whose position is expressed by an unknown function $\bb:[0,T)\mapsto \mathbb R^d$, $\bb(t)\times\be_1  = 0$ (i.e. $\Omega_t =\Omega_0 + \bb(t)$), is joined by a spring to a point whose position is at ${\bf f}(t)$, where ${\bf f}:[0,T)\mapsto \mathbb R^d$, ${\bf f}(t)\times \be_1 = 0$ is given -- this means that the oscillations are forced as one end of the spring possesses a motion prescribed by ${\bf f}$. 

The movement of the fluid inside a container is given by the Navier-Stokes equations, i.e.
\begin{equation}\label{NS.unsteady}
\begin{split}
\pat (\varrho \bw )  + \diver (\varrho \bw \otimes \bw) - \diver T & = 0 \ \mbox{in}\ \Omega_t\\
\pat \varrho + \diver \varrho \bw &= 0 \ \mbox{in}\ \Omega_t.
\end{split}
\end{equation}
Here $\bw:[0,T)\times \Omega_t\mapsto \mathbb R^d$ is a velocity of the fluid and $\varrho:[0,T)\times \Omega_t\mapsto \mathbb R$ is its density. We assume the no-slip boundary condition, i.e. 
\begin{equation*}%\label{no.slip.unsteady}
\bw{\restriction_{\partial \Omega_t}} = \dot \bb.
\end{equation*}
The stress tensor $T$ is given as
\begin{equation}\label{stress.tensor}
T = T(\bw,\varrho) = S(\nabla \bw) - I p(\varrho)
\end{equation}
where 
\begin{equation}\label{deviatoric.stress}
S(\nabla \vw) = \mu D\bw +(\lambda +\mu)I \diver \bw,
\end{equation}
$I$ is an  identity $d\times d$ matrix and $Dw = \frac 12 \left(\nabla w + (\nabla w)^T\right)$, $\mu>0$, $\lambda + \frac 23 \mu>0$,
and
\begin{equation}\label{tlak}
p(\varrho) = a \varrho^\gamma
\end{equation}
for some $a,\ \gamma\in \mathbb R^+$ specified later.

The force caused by the spring, whose given stiffness is $k>0$, is 
$$
\bF(t) = -(\bb(t) - {\bf f}(t)) k
$$
By Newton law, this cause a change in linear momentum and by the transport theorem we write
$$
\bF = \partial_t\int_{\Omega_t} \varrho \bw \dx = \int_{\Omega_t} \pat \varrho \bw + \diver(\varrho \bw\otimes \bw)\dx
$$
and we use \eqref{NS.unsteady} with the integration by parts in order to deduce the following coupling between the fluid and the movement of the container:
\begin{equation*}%\label{Newton.law}
-(\bb(t) - {\bf f}(t)) k = \left(\int_{\partial \Omega_t} T \bn {\rm d}S_x\cdot \ve_1\right) \ve_1.
\end{equation*}
Note that the movement is allowed only in $\ve_1$ direction.

\subsection{Rigid body case, incompressible case}%\label{rbcic}
In the case of the rigid body (it is sufficient to take $\varrho \equiv 1$, $\bw\equiv \dot \bb$ and forget about \eqref{NS.unsteady}$_1$) one end up with an equation
$$
\ddot \bb + \frac k M \bb = \frac kM {\bf f}(t)
$$
where $M$ is a total mass of the rigid body at the end of the spring.

Let assume that ${\bf f}(t)\cdot \be_1 = \sin (\omega t)$. As this equation is a second-order linear differential equation with constant coefficients, we may use a standard theory (Duhamel formula) to deduce that for $\omega^2\neq \frac kM$ the solutions have form
$$
\bb\cdot \be_1 = c_1 \sin\left(\sqrt{\frac kM} t\right) + c_2\cos \left(\sqrt{\frac kM} t\right) + \frac{k}{k-M\omega^2}\sin (\omega t)
$$
where $c_1,\ c_2\in \mathbb R$ are arbitrary constants. 
However, there appears a resonance for $\omega^2 = \frac kM$ -- the solution in such case is 
\begin{equation*}%\label{resonance}
\bb\cdot \be_1 =-\frac\omega{2} t \cos (\omega t) +  c_1 \sin(\omega t) + c_2\cos (\omega t)
\end{equation*}
and one can see that such solution is unbounded.

Consider now an incompressible fluid. Let $\varrho \equiv 1$ and $\diver \bw = 0$ in \eqref{NS.unsteady}. Also assume that $p$ is an independent unknown. In that case one solution is $\bw = \dot \bb$, pressure is such that 
$$
\nabla p = \dot \bb
$$
and we end up with the same system as in the rigid body case. Roughly speaking, the force acting on the fluid is potential and it does not cause the fluid to move. Therefore, there appears no damping. 

However, this phenomena does not happen once the fluid is compressible. Indeed, the assumptions $\vw = \dot \vb$  and $\nabla p(\varrho) = \dot \vb$ yield a contradiction -- the second equation gives $\varrho(t,x) = \varrho(t)$ and thus $\pat \varrho$ is nonzero function of time which cannot be compensated by $\diver \varrho \vw = \nabla \varrho \vw + \varrho \diver \vw$ and thus \eqref{NS.unsteady}$_2$ is not fulfill. Consequently, there must appear a non-trivial flux.

\section{Formulation}
It is convenient to rewrite \eqref{NS.unsteady} such that the resulting system is considered on a bounded domain. Therefore we introduce a new variable $x = y-\bb(t)$ and a new velocity $\bu(t,x) = \bw(t,x+ b(t))$. Then\footnote{Note that $\diver(\varrho \bu \otimes \bv) = \partial_i(\varrho \bu_j \bv_i)$ where the summation convention is used}
\begin{equation}\label{NS.fixed}
\begin{split}
\pat (\varrho \bu) + \diver (\varrho \bu\otimes \bv) - \diver T & = 0 \ \mbox{in }\Omega_0\\
\pat \varrho + \diver \varrho \bv & = 0\ \mbox{in }\Omega_0
\end{split}
\end{equation} 
where $\bv = \bu-\dot\bb_t$. Furthermore, the above system is complemented by the no-slip boundary condition
\begin{equation}\label{no.slip.fixed}
\bu{\restriction_{\partial\Omega}} = \dot \bb
\end{equation}
and by the Newton law
\begin{equation}\label{NL.fixed}
-\left((\bb(t) - {\bf f}(t))k\right)\cdot \ve_1 = \left(\int_{\partial \Omega} T\bn\ {\rm d}Sx\right)\cdot\ve_1.
\end{equation}
Recall that $\vb = b \ve_1$ for some real function $b$.
Here $T = T(\vu,\varrho)$ is given by \eqref{stress.tensor}, \eqref{deviatoric.stress} and \eqref{tlak} and $\bn$ denotes unit normal aiming outside the fluid domain. We would like to point out that $T$ itself is a symmetric tensor and it depends on the symmetric part of $\nabla\vu$. Therefore, $T(\vu,p) = T(\vv,p)$ and $T:\nabla\vu = T:\nabla\vv$.

We multiply formally \eqref{NS.fixed}$_1$ by $u$ in order to obtain the energy (in)equality which posses the following form:
\begin{equation}\label{energy.equality}
\pat\int_{\Omega}\frac 12 \varrho |\bu|^2\dx + \pat \int_\Omega \frac a{\gamma-1} \varrho^\gamma \dx + \int_\Omega S(\nabla \bv):\nabla \bv \dx + \frac k2 \pat |\bb(t)|^2 \leq k \dot \bb (t) {\bf f}(t).
\end{equation}
As usual, smooth solutions satisfy \eqref{energy.equality} with the equality sign. Now, we are ready to define the notion of weak solutions.

\begin{definition}
We say, that $(\varrho,\bu)\in L^\infty((0,T),L^\gamma(\Omega))\times L^2((0,T),W^{1,2}(\Omega))$ is a weak solution to \eqref{NS.fixed}, \eqref{no.slip.fixed} and \eqref{NL.fixed} if
\begin{itemize}
\item The continuity equation is fulfilled in a weak sense, i.e.
$$
\int_0^T\int_{\Omega} \varrho \pat \varphi  + \varrho \bv \nabla \varphi \ \dx \dt + \int_\Omega \varrho_0\varphi(0,\cdot)\ \dx = 0
$$
for all $\varphi \in C^\infty_c([0,T)\times \overline \Omega)$.
\item The momentum equation is fulfilled in a weak sense, i.e.
$$
\int_0^T\int_\Omega \varrho \bu \pat \varphi  + \varrho \bu \otimes \bv \nabla \varphi + T:\nabla \varphi\ \dx \dt  + \int_0^T \varphi{\restriction_{\partial \Omega}} (\bb(t) - {\bf f}(t))k\ \dt+ \int_\Omega (\varrho \bu)_0\varphi(0,\cdot)\ {\rm d}x = 0
$$
holds for all $\varphi \in C^\infty_c([0,T)\times \overline \Omega)$ such that there exists a function $\bb_\varphi\in C^\infty_c([0,T))$ such that $\varphi{\restriction_{\partial \Omega}} = \bb_\varphi$.
\item $\bu|_{\partial \Omega} = \dot\bb(t)$ for some $\bb\in W^{1,\infty}(0,T)$, $\vb\times\ve_1 = 0$.
\item The energy inequality \eqref{NS.fixed} holds.
\end{itemize}
\end{definition}

Note that this definition is sufficient in the sense that once $(\varrho, \bu)$ is a sufficiently smooth weak solution, it solves \eqref{NS.fixed} pointwisely. Now we are ready to state the main theorem of this paper.

\begin{theorem}
Assume that $\Omega$ is a $C^{2+\nu}$, $\nu>0$ domain and $\gamma>3/2$. For any given $T>0$ and $\varrho_0\in L^\gamma(\Omega)$, $(\varrho \bu)_0\in L^{2\gamma/(\gamma+1)}(\Omega)$ there exists a weak solution to \eqref{NS.fixed}, \eqref{no.slip.fixed} and \eqref{NL.fixed}.
\end{theorem}

This theorem is proven by means of the nowadays standard Feireisl-Lions theory (see \cite{FeNoPe}). First, we introduce an approximate system possessing certain regularity properties which allows to prove the existence result -- this is done in Section \ref{app.system}. Than, we tend with all the approximations to zero in order to reconstruct the original system. This is a content of Section \ref{tend.to.zero}

\section{Approximate system} \label{app.system}

There are two approximations. First, we introduce artificial viscosity in the continuity equation (this is represented by $\varepsilon>0$) and then we add artificial pressure to the momentum equation (represented by $\delta>0$). The system we have in mind is
\begin{equation}
\begin{split}\label{eq.app.sys}
\pat \varrho + \diver(\varrho\vv) &= \varepsilon\Delta\varrho\\
\pat(\varrho\vu) + \diver (\varrho \vu\otimes\vv) - \diver S(\vv) + \nabla p(\varrho) +  \nabla\delta \varrho^8 + \varepsilon (\nabla \varrho \nabla)\vu &= 0
\end{split}
\end{equation}
and it is complemented with 
\begin{equation}\label{eq.bound.app}
\vu{\restriction_{\partial\Omega}} = \dot\vb,\quad \frac{\partial\varrho{\restriction_{\partial\Omega}}}{\partial \vn} = 0
\end{equation}
and
\begin{equation}\label{eq.newton.app}
-(\vb(t) - \vf(t))k = \int_{\partial\Omega} T\vn \ {\rm d}S_x.
\end{equation}
Moreover, we assume
$$
\varrho(0,\cdot) = \varrho_{0,\delta}(\cdot),\ (\varrho \vu)(0,\cdot) = (\varrho\vu)_{0}(\cdot)
$$ 
where $\varrho_{0,\delta}$  is such that
\begin{equation}
\varrho_{0,\delta}\in C^{2+\nu}(\Omega),\  0\leq \underline\varrho\leq \varrho_{0,\delta}\leq \overline\varrho<\infty \label{init.rho}
\end{equation}
and
$$
\varrho_{0,\delta}\to \varrho_0\ \mbox{strongly in }L^\gamma(\Omega)\ \mbox{as }\delta\to 0.
$$
Of course, constants appearing in \eqref{init.rho} are supposed to depend on $\delta$.

\begin{theorem}\label{app.exist}
Assume $\Omega$ is in class $C^{2+\nu}$ for some $\nu>0$. There exists a solution\footnote{Hereinafter we use an abbreviation ${\max\{8,\gamma\}} = \beta$} $(\varrho,\vu)\in L^\infty((0,T),L^\beta(\Omega))\times L^2((0,T),W^{1,2}(\Omega))$ to \eqref{eq.app.sys}, \eqref{eq.bound.app} and \eqref{eq.newton.app} such that there exists $\vb\in L^\infty$ fulfilling
$
\vu{\restriction_{\partial\Omega}} = \vb
$, $\vv := \vu - \vb\in L^2((0,T),W^{1,2}_0(\Omega))$ and $\varrho\in L^2((0,T),W^{2,2}(\Omega))$. Moreover, the solution satisfies the energy inequality in a form
\begin{multline}\label{ene.ineq.app}
\pat \int_\Omega \frac 12 \varrho |\vu|^2\ {\rm d}x + \pat \int_\Omega \frac a{\gamma-1}\varrho^\gamma + \frac \delta7\varrho^8\ {\rm d}x + \int_\Omega S(\nabla\vv):\nabla\vv\ {\rm d}x\\
 + \int_\Omega \varepsilon a\gamma\varrho^{\gamma-1}|\nabla\varrho|^2 + 8 \varepsilon\delta \varrho^6|\nabla\varrho|^2\ {\rm d}x
+ \frac k2 \pat |\vb(t)|^2  
\leq k\cdot\vb(t)\vf(t)
\end{multline}
\end{theorem}

\begin{proof}
Our aim is to solve \eqref{eq.app.sys} by means of the Galerkin approximations. We consider an $n-$dimensional space spanned by egienvectors of the Laplace operator, i.e.,
$$
X_n = [{\mbox{span}}\{\psi_j\}]^3
$$
where
$$
 -\Delta \psi_n = \lambda_n \psi_n \ \mbox{on }\Omega,\ \psi_n{\restriction_{\partial\Omega}} = 0.
$$
Given $\varrho_n$, we look for an approximate solution $\vv_n,\ \vb_n$ where $\vv_n \in X_n$ and
$$
\vb_n(t) = \vf(t) - \left(\frac 1k \int_{\partial\Omega} T\vn \ {\rm d}S_x\cdot\ve_1 \right) \ve_1
$$
(note that $T$ depends only on $\varrho$ and $\vv_n$). We are looking for a function $\vu_n = \vv_n + \vb_n$  satisfying
\begin{multline}\label{duhamel}
\int_\Omega \varrho_n \vu_n\psi\ {\rm d}x - \int_\Omega (\varrho\vu)_0 \psi\ {\rm d}x \\
 = \int_0^t\int_\Omega \left(\diver S(\vv_n) - \nabla p(\varrho_n) - \delta\nabla \varrho_n^8 - \varepsilon \nabla\varrho_n\nabla\vu_n - \diver(\varrho_n \vu_n  \otimes\vv_n)\right) \psi \ {\rm d}x{\rm d}t
\end{multline}
for all $\psi\in \mathbb R^3\oplus X_n$, $\psi{\restriction_{\partial\Omega}} \times \ve_1 = 0$ (we will refer to such space by notation $\mathbb R\times X_n$ even though it is not completely correct). Note that $\vu_n$ belongs to a (finite-dimensional) space $\mathbb R\times X_n$. Here we would like to note that \eqref{duhamel} actually contains equations \eqref{eq.app.sys}$_2$ and \eqref{eq.newton.app}. %Indeed, we perform integration by parts in order to obtain
%\begin{multline*}
%\int_\Omega \varrho_n \vu_n\psi\ {\rm d}x - \int_\Omega (\varrho\vu)_0 \psi\ {\rm d}x \\
% = \int_0^t\int_\Omega - S(\vv_n):\nabla\psi - ( p(\varrho_n) + \delta \varrho_n^8)\diver\psi - \varepsilon \nabla\varrho_n\nabla\vu_n\psi + (\varrho_n \vu_n  \otimes\vv_n) \nabla \psi \ {\rm d}x{\rm d}t
%\end{multline*}

Next, we define $\mathcal M_{\varrho_n}:\mathbb R\times X_n\mapsto \mathbb R\times X_n$ as
$$
\mathcal M_{\varrho_n}(\vv) = \vw \iff \int_\Omega \varrho_n \vv \psi\ {\rm d}x = \int_\Omega \vw \psi {\rm d}x\ \forall \psi \in \mathbb R^3\times X_n.
$$
This mapping is invertible assuming $\frac 1{\varrho_n}\in L^\infty(\Omega)$ and it satisfies
\begin{equation*}
\begin{split}
\|\mathcal M_{\varrho_n}^{-1}\|&\leq \left\|\frac 1{\varrho_n}\right\|_{L^\infty(\Omega)}\\
\mathcal M^{-1}_{\varrho_n} - \mathcal M^{-1}_{\varrho'_n}& = \mathcal M^{-1}_{\varrho_n'}(\mathcal M_{\varrho'_n} - \mathcal M_{\varrho_n}) \mathcal M^{-1}_{\varrho_n}
\end{split}
\end{equation*}
for any $\varrho_n$ and $\varrho'_n$. We deduce from \eqref{duhamel} that
$$
\vu_n = \mathcal M^{-1}_{\varrho_n} \left(\int_0^t \diver S(\vv_n) - \nabla p(\varrho_n) - \delta\nabla (\varrho_n^8) - \varepsilon\nabla \varrho_n\nabla\vu_n -\diver (\varrho_n\vu_n\otimes\vv_n)   \ {\rm d}t + (\varrho\vu)_0\right).
$$
We can deduce
\begin{equation}\label{vv.fix.point}
\vv_n = \mathcal P\left(\mathcal M^{-1}_{\varrho_n} \left(\int_0^t \diver S(\vv_n) - \nabla p(\varrho_n) - \delta\nabla (\varrho_n^8) - \varepsilon\nabla \varrho_n\nabla\vu_n -\diver (\varrho_n\vu_n\otimes\vv_n)   \ {\rm d}t + (\varrho\vu)_0\right) \right)
\end{equation}
where $\mathcal P: \mathbb R\times X_n \mapsto X_n$ is an orthogonal projection. Next, given $\vv_n$ and $\varrho_0$ there is a unique solution to
$$
\pat \varrho + \diver(\varrho\vv_n) = \varepsilon \Delta\varrho,\ \frac{\partial \varrho{\restriction_{\partial \Omega}}}{\partial \vn} = 0,\ \varrho(0,\cdot) = \varrho_{0,\delta}(\cdot).
$$
Note that $\|\vv\|_{W^{1,\infty}}\leq c$ is required to get $\left\|\frac 1{\varrho_n}\right\|\leq c$ which yields the existence of $\mathcal M^{-1}_{\varrho_n}$. Moreover, the mapping $S: C([0,T],X_n)\mapsto C([0,T], W^{1,2}(\Omega))$, $S(\vv_n) = \varrho_n$ s Lipschitz -- see \cite[Lemma 2.2]{FeNoPe}.

In particular, $\varrho_n$, $\vb_n$ and $\vu_n$ are functions of $\vv_n$ and thus \eqref{vv.fix.point} is solvable (at least on short time interval) by use of the Banach fix-point argument.

Further, we differentiate \eqref{duhamel} with respect to $t$ and we use $\vu_n(t)$ as a test function to obtain
\begin{multline}\label{first.energy}
\pat \left(\int_\Omega \frac 12 \varrho_n |\vu_n|^2 + \frac a{\gamma-1}\varrho_n^\gamma + \frac \delta 7 \varrho_n^8\ {\rm d}x\right) + \int_\Omega S(\nabla\vv_n):\nabla\vv_n  + \left(\varepsilon a \gamma \varrho_n^{\gamma-2} + 8\varepsilon \delta\varrho_n^6\right)|\nabla\varrho_n|^2 \ {\rm d}x\\ + \frac k2 \pat |\vb|^2  \leq k\dot\vb \vf.
\end{multline}
We use the Gronwall inequality to deduce\footnote{Hereinafter we use the following notation: $\|\cdot\|_{L^p}$ is a norm in $L^p(\Omega)$, $\|\cdot\|_{W^{k,p}}$ is a norm in $W^{k,p}(\Omega)$,  $\|\cdot\|_{L^pL^q}$ is norm in $L^p((0,T),L^q(\Omega))$ and $\|\cdot\|_{L^pW^{k,p}}$ is a norm in $L^p((0,T),W^{k,p}(\Omega)$.}
$$
\sup_t \|\varrho_n|\vu_n|^2\|_{L^1} + \sup_t \|\varrho_n\|_{L^\beta} + \|\vu_n\|_{L^2W^{1,2}}\leq c,
$$
where $\beta = \max\{\gamma,8\}$. Since all norms on the finite-dimensional space are equivalent we deduce a uniform bound
$$
\sup_t \left(\|\vu_n(t)\|_{L^\infty} + \|\nabla\vu_n(t)\|_{L^\infty}\right) \leq c(n,T,\varrho_0,(\varrho\vu)_0)
$$
which allow to extend the solution to \eqref{duhamel} to the whole given interval $(0,T)$. 

We are going to pass with $n\to\infty$. We deduce from \eqref{first.energy} 
\begin{equation}
\begin{split}\label{first.est}
\sup_{t\in(0,T)} \|\varrho_n\|_{L^\beta} &\leq c\\
\sup_{t\in(0,T)} \|\varrho_n|\vu_n|^2\|_{L^1} & \leq c\\
\|\vu_n\|_{L^2 W^{1,2}} & \leq c\\
\varepsilon \|\nabla\varrho_n\|_{L^2W^{1,2}}& \leq c
\end{split}
\end{equation}
and, consequently, also these convergencies (up to a subsequence)
\begin{equation*}
\begin{split}
\varrho_n&\to\varrho\ \mbox{weakly* in }L^\infty((0,T),L^\beta(\Omega))\\
\vu_n&\to \vu \ \mbox{weakly in } L^2((0,T),W^{1,2}(\Omega)).
\end{split}
\end{equation*}
The Aubin-Lions lemma together with \eqref{first.est}$_{1,4}$ yields $\varrho_n\to\varrho$ strongly in, say $L^4((0,T)\times \Omega)$. This strong convergence then gives
$$
\varrho_n\vu_n \to \varrho \vu\ \mbox{ weakly* in } L^\infty((0,T),L^{2\beta(\beta+1)}(\Omega)).
$$
Further, \cite[Lemma 2.4]{FeNoPe} yields existence of $r>1$ and $q>2$ such that
\begin{equation*}
\|\pat\varrho_n\|_{L^rL^r} + \|\Delta\varrho_n\|_{L^rL^r} + \|\nabla\varrho_n\|_{L^qL^2} \leq c
\end{equation*}
with $c$ independent of $n$. As a matter of fact, $\varrho$ and $\vv$ solves
$$
\pat\varrho + \diver(\varrho\vv) = \varepsilon \Delta\varrho.
$$
We test continuum equations by $\varrho_n$ ($\varrho$ respectively) to obtain
$$
\|\varrho_n(t)\|_{L^2}^2 + 2\varepsilon \int_0^t\|\nabla \varrho_n\|_{L^2}\ {\rm d}t = -\int_0^t\int_\Omega \diver \vv_n |\varrho_n|^2 \ {\rm d}x{\rm d}t + \|\varrho_0\|_{L^2}^2
$$
and
$$
\|\varrho(t)\|_{L^2}^2 + 2\varepsilon \int_0^t\|\nabla \varrho\|_{L^2}\ {\rm d}t = -\int_0^t\int_\Omega \diver \vv |\varrho|^2 \ {\rm d}x{\rm d}t + \|\varrho_0\|_{L^2}^2
$$
This with already mentioned convergencies yield
$$
\nabla\varrho_n \to \nabla\varrho\ \mbox{ strongly in }L^2((0,T)\times\Omega).
$$
This gives $\nabla\varrho_n\nabla\vu_n\to \nabla\varrho \nabla \vu$ in the sense of distribution and, consequently, we deduce that the limit functions $\varrho$ and $\vu$ satisfies \eqref{eq.app.sys}. More precisely, \eqref{eq.app.sys}$_1$ is fulfilled in the weak sense (and also in the renormalized weak sense due to the regularity of $\varrho$ and $\vv$) and 
\begin{multline*}
\int_0^T\int_\Omega \varrho \vu \pat \varphi + \varrho\vu \otimes\vv \nabla\varphi + S(\vv):\nabla\varphi - (p(\varrho) + \delta\varrho^8)\diver \varphi + \varepsilon \nabla \varrho_n \vu_n \varphi \ {\rm d}x{\rm d}t\\
 + \int_0^T\varphi{\restriction_\Omega}(\vb(t) - \vf(t)) k\ {\rm d}t + \int_\Omega (\varrho\vu)_0 \varphi(0,\cdot)\ {\rm d}x = 0
\end{multline*}
for all $\varphi\in C^\infty_c([0,T)\times \overline\Omega)$ for which there is a function $\vb_\varphi \in  C^\infty_c([0,T))$, $\vb_\varphi\times \ve_1 = 0$ such that $\varphi{\restriction_{\partial\Omega}} = \vb_\varphi$.

\end{proof}

\section{Vanishing of the approximations} \label{tend.to.zero}
In this Section we perform (succesively) limits $\varepsilon\to 0$ and $\delta\to 0$ in \eqref{eq.app.sys}--\eqref{eq.newton.app}.
\subsection{Limit $\varepsilon\to 0$}
Let $\varrho_\varepsilon$ and $\vu_\varepsilon$ be the solution constructed in Theorem \ref{app.exist}. We integrate \eqref{ene.ineq.app} over $(0,t)\subset (0,T)$ and we use the Gronwall inequality to obtain the following set of estimates
\begin{equation}
\begin{split} \label{ene.est.1}
{\rm esssup}_{t\in (0,T)}\|\varrho_\varepsilon\|_{L^\beta} &\leq c\\
\|\vv_\varepsilon\|_{L^2W^{1,2}} & \leq c\\
\|\bb_\varepsilon\|_{L^\infty(0,T)} & \leq c\\
{\rm esssup}_{t\in (0,T)}\|\varrho_\varepsilon|\vu_\varepsilon|^2\|_{L^1} & \leq c
\end{split}
\end{equation}
Moreover, we integrate \eqref{eq.app.sys} over $\Omega$ to get
\begin{equation}\label{mass.conservation}
\int_\Omega\varrho_\varepsilon(t,\cdot)\ {\rm d}x = \int_\Omega \varrho_{0,\delta}
\end{equation}
for all $t\in (0,T)$. Note also that \eqref{ene.est.1}$_{1,2,4}$, \eqref{mass.conservation} and \cite[Lemma 3.2]{feireisl} yields
$$
\|\vu_\varepsilon\|_{L^2W^{1,2}}\leq c.
$$
We multiply \eqref{eq.app.sys}$_1$ by $\varrho$ to deduce
$$\int_{\Omega} \frac 12 \pat |\varrho_\varepsilon|^2\ {\rm d}x + \int_\Omega\varepsilon |\nabla \varrho_\varepsilon|^2\ {\rm d}x = \int_\Omega \varrho_\varepsilon\vu_\varepsilon \nabla\varrho_\varepsilon\ {\rm d}x
$$
which together with \eqref{ene.est.1} and $\int_\Omega\varrho_\varepsilon\vu_\varepsilon\nabla\varrho_\varepsilon\ {\rm d}x = -\frac 12\int_\Omega \diver \vu_\varepsilon |\varrho_\varepsilon|^2\ {\rm d}x$ yield
\begin{equation}\label{est.rho.regular}
\varepsilon\|\nabla\varrho\|_{L^2(L^2)}^2 \leq c.
\end{equation}
It is worth to mention that constants appearing on the right hand side of the estimates \eqref{ene.est.1} and \eqref{est.rho.regular} are independent of $\varepsilon$. Next, we would like to deduce higher integrability of the density. We use $\varphi = \mathcal B(\varrho - (\varrho)_\Omega)$ as a test function in \eqref{eq.app.sys}$_2$. Here $\mathcal B$ stands for the Bogovski operator and $(\varrho)_\Omega = \frac 1{|\Omega|} \int_\Omega\varrho \ {\rm d}x$. We deduce that
\begin{multline}
\int_0^T\int_\Omega (p(\varrho_\varepsilon) + \delta\varrho_\varepsilon^8)\varrho_\varepsilon\ {\rm d}x = \int_0^T\int_\Omega(p(\varrho_\varepsilon) + \delta\varrho_\varepsilon^8)(\varrho_\varepsilon)_\Omega\ {\rm d}x{\rm d}t \\
 -\int_\Omega \left(\varrho_\varepsilon\vu_\varepsilon\mathcal B(\varrho_\varepsilon - (\varrho_\varepsilon)_\Omega)(T,\cdot) - (\varrho\vu)_0\mathcal B(\varrho_0 - (\varrho_0)_\Omega)(\cdot)\right)\ {\rm d}x + \int_0^T\int_\Omega\varrho_\varepsilon\vu_\varepsilon \pat\mathcal B(\varrho_\varepsilon - (\varrho_\varepsilon)_\Omega)\ {\rm d}x{\rm d}t\\
 + \int_0^T\int_\Omega \varrho_\varepsilon\vv_\varepsilon\otimes\vu_\varepsilon\nabla\mathcal B(\varrho_\varepsilon - (\varrho_\varepsilon)_{\Omega})\ {\rm d}x{\rm d}t - \int_0^T\int_\Omega S(\nabla\vv_\varepsilon) \nabla \mathcal B(\varrho_\varepsilon - (\varrho_\varepsilon)_\Omega)\ {\rm d}x{\rm d}t\\
 - \varepsilon\int_0^T\int_{\Omega} \nabla\varrho_\varepsilon \nabla\vu_\varepsilon \mathcal B(\varrho_\varepsilon - (\varrho_\varepsilon)_\Omega)\ {\rm d}x{\rm d}t.\label{eq.bog.test}
\end{multline}

Recall that $\mathcal B$ is a bounded linear operator which maps $L^p(\Omega)\mapsto W^{1,p}(\Omega)$ for all $p\in (1,\infty)$. Since $\varrho_\varepsilon\in L^\infty((0,T),L^\beta(\Omega))$ uniformly in $\varepsilon$ and $(\varrho_\varepsilon)_{\Omega}$ is a constant due to \eqref{mass.conservation}, we are able to deduce that the right hand side of \eqref{eq.bog.test} is bounded independently of $\varepsilon$. Indeed, the only troublemaker is the term $\int_0^T\int_\Omega \varrho_\varepsilon\vu_\varepsilon\pat\mathcal B(\varrho_\varepsilon - (\varrho_\varepsilon)_{\Omega})\ {\rm d}x{\rm d}t$. However, we use the linearity of $\mathcal B$ together with the continuity equation in order to deduce
\begin{multline*}
\int_0^T\int_\Omega \varrho_\varepsilon\vu_\varepsilon \pat \mathcal B(\varrho_\varepsilon - (\varrho_\varepsilon)_\Omega)\ {\rm d}x{\rm d}t  = \int_0^T\int_\Omega \varrho_\varepsilon \vu_\varepsilon \mathcal B(\pat\varrho)\ {\rm d}x{\rm d}t = -\int_0^T\int_\Omega\varrho_\varepsilon\vu_\varepsilon \mathcal B(\diver(\varrho_\varepsilon\vu_\varepsilon)\ {\rm d}x{\rm d}t\\
 \leq \int_0^T \|\varrho\vu\|_{2}^2\ {\rm d}t\leq \int_0^T \|\varrho\|_{L^3}^2 \|\vu\|_{L^6}^2\leq c
\end{multline*}
with $c$ independent of $\varepsilon$. Consequently, 
\begin{equation}\label{eq.rho.int}
\|\varrho_\varepsilon\|_{L^{\beta+1}((0,T)\times\Omega)}\leq c.
\end{equation}

Therefore, we may extract a subsequence of solutions such that
\begin{equation*}
\begin{split}
\varrho_\varepsilon & \to \varrho \ \mbox{ weakly* in }L^\infty((0,T),L^{\beta}(\Omega))\\
\vu_\varepsilon &\to \vu \ \mbox{ weakly in }L^2((0,T),W^{1,2}(\Omega))\\
\vb_\varepsilon &\to \vb \ \mbox{ weakly* in }L^\infty(0,T).
\end{split}
\end{equation*}
Note also that we immediately obtain also $\vv_\varepsilon\to \vv$ weakly in $L^2((0,T),W^{1,2}(\Omega))$.

Since $L^\beta(\Omega)$ is compactly embedded in $W^{-1,2}(\Omega)$ and, from \eqref{eq.app.sys}$_1$ we deduce $\pat\varrho\in L^\infty((0,T), W^{-1,2\beta/(\beta+1)}(\Omega))$ we get (see \cite[Lemma 6.2]{NoSt})
\begin{equation}\label{eq.rho.str}
\begin{split}
\varrho_\varepsilon &\to \varrho\ \mbox{ strongly in }C((0,T),L^\beta_{weak}(\Omega))\\
\varrho_\varepsilon &\to \varrho\ \mbox{ strongly in }L^p((0,T),W^{-1,2}(\Omega))
\end{split}
\end{equation}
for $p\in (1,\infty)$ and, consequently
$$
\varrho_\varepsilon \vu_\varepsilon \to \varrho\vu \ \mbox{ weakly* in }L^\infty((0,T), L^{2\beta/(\beta+1)}(\Omega)).
$$
The momentum equation together with \eqref{ene.est.1} yield a uniform continuity of $\varrho_\varepsilon\vu_\varepsilon$ in $W^{-1,(\beta+1)/\beta}(\Omega)$ and we argue similarly as before to deduce
\begin{equation}\label{eq.mom.str}
\begin{split}
\varrho_\varepsilon \vu_\varepsilon &\to \varrho\vu\ \mbox{ strongly in }C((0,T),L^{2\beta/(\beta+1)}_{weak}(\Omega))\\
\varrho_\varepsilon \vu_\varepsilon &\to \varrho\vu\ \mbox{ strongly in }L^p((0,T),W^{-1,2}(\Omega)).
\end{split}
\end{equation}
The deduced convergences then yields that $\varrho$ and $\vu$ satisfies
\begin{equation*}
\int_0^T\int_\Omega \varrho \pat \varphi + \varrho\vv \nabla\varphi\ {\rm d}x{\rm d}t + \int_\Omega \varrho_{0,\delta} \varphi(0,\cdot)\ {\rm d}x = 0
\end{equation*}
for all $\varphi \in C^\infty_c([0,T)\times \overline\Omega$ and
\begin{multline}\label{eq.mom.befe}
\int_0^T \int_\Omega \varrho \vu \pat \varphi + \varrho \vu \otimes\vv \nabla\varphi + S(\nabla\vv):\nabla\varphi - \overline{(p(\varrho)+\delta\varrho^8)}\diver\varphi\ {\rm d}x{\rm d}t + \int_0^T  \varphi{\restriction_{\partial\Omega}}(\vb(t) - \vf(t)) k\ {\rm d}t \\
 + \int_\Omega (\varrho\vu)_0 \varphi(0,\cdot)\ {\rm d}x = 0
\end{multline}
for all $\varphi \in C^\infty_c([0,T)\times\overline\Omega)$ such that there exists a function $\vb_\varphi\in C^\infty_c([0,T))$, $\vb_\varphi \times \ve_1 = 0$ fulfilling $\varphi{\restriction_{\partial\Omega}} = \vb_\varphi$. Here we use the notation $\overline{f(\varrho)}$ to denote the weak limit of $f(\varrho_\varepsilon)$.

We are going to prove that 
\begin{equation}\label{eq.tlak}
\overline{(p(\varrho) + \delta\varrho^8)} = p(\varrho) + \delta\varrho^8.
\end{equation}
The first step is to prove that
\begin{lemma}
It holds that
\begin{equation}\label{eq.e.v.f}
\int_0^T\int_\Omega \left(p(\varrho_\varepsilon) + \delta\varrho_\varepsilon^8 - (\lambda + 2\mu)\diver\vv_\varepsilon\right)\varrho_\varepsilon \varphi\ {\rm d}x{\rm d}t \to \int_0^T\int_\Omega\left(\overline{(p(\varrho) + \delta\varrho^8)} - (\lambda + 2\mu)\diver \vv\right)\varrho \varphi\ {\rm d}x{\rm d}t
\end{equation}
for all $\varphi\in C^\infty_c((0,T)\times\Omega)$.
\end{lemma}
\begin{proof}
We take $\Phi_j = \varphi A_j(\varrho_\varepsilon)$ as a test function in \eqref{eq.app.sys}$_2$. Here $A_j(\varrho_\varepsilon)$ is an inverse to $\diver $ and its Fourier symbol is
$$
A_j(\xi) = \frac{-i \xi_j}{|\xi|^2}
$$
(roughly, $A(\varrho_\varepsilon) = \nabla \Delta^{-1} (\varrho_\varepsilon)$).  Further, $\varphi$ is a real-valued smooth function with compact support in $(0,T)\times \Omega$. We also smphasise that $A$ has a symmetric gradient, namely $\partial_i A_j = \partial_j A_i$. Moreover, we assume $\varrho$ is extended by $0$ outside of $\Omega$. After a cumbersome calculation we arrive at%\footnote{Summation convention is used throughout this paper.}
\begin{multline}
\int_0^T \int_\Omega \varphi (p(\varrho_\varepsilon) + \delta \varrho_\varepsilon^8 - (\lambda + 2\mu)\diver \vv_\varepsilon)\varrho_\varepsilon \ {\rm d}x{\rm d}t\\
\label{rovnice.410}%
 = \int_0^T\int_\Omega (\lambda+ \mu) \diver \vv_\varepsilon \nabla \varphi A(\varrho_\varepsilon)\ {\rm d}x{\rm d}t - \int_0^T\int_\Omega (p(\varrho_\varepsilon) + \delta \varrho^8)\nabla \varphi A(\varrho_\varepsilon)\ {\rm d}x{\rm d}t\\
 + \mu \int_0^T \int_\Omega \nabla \vu_\varepsilon \nabla \varphi A(\varrho_\varepsilon)\ {\rm d}x{\rm d}t - \mu \int_0^T \int_\Omega \vv_\varepsilon \nabla \varphi \nabla A(\varrho_\varepsilon) \ {\rm d}x{\rm d}t + \mu \int_0^T \int_\Omega  \vv_\varepsilon \nabla \varphi \varrho_\varepsilon \ {\rm d}x{\rm d}t\\
+ \varepsilon \int_0^T\int_\Omega \nabla \vu_\varepsilon \nabla \varrho_\varepsilon \varphi A(\varrho_\varepsilon)\ {\rm d}x{\rm d}t - \varepsilon \int_0^T\int_\Omega \varrho_\varepsilon \vu_\varepsilon \varphi A(\diver (\chi_\Omega \nabla \varrho_\varepsilon)) \ {\rm d}x{\rm d}t - \int_0^T\int_\Omega \varrho_\varepsilon \vv_\varepsilon \otimes \vu_\varepsilon \nabla \varphi A(\varrho_\varepsilon)\ {\rm d}x{\rm d}t\\
-\int_0^T\int_\Omega \varrho_\varepsilon \vu_\varepsilon \pat \varphi A(\varrho_\varepsilon)\ {\rm d}x{\rm d}t
+\int_0^T\int_\Omega \varphi \varrho_\varepsilon \vu_\varepsilon (\partial_i A_j(\varrho_\varepsilon \vv_{i\varepsilon}) - \vv_{i\varepsilon}\partial_i A_j(\varrho_\varepsilon))\ {\rm d}x{\rm d}t
\end{multline}

We also use $\varphi A(\varrho)$ as a test function in \eqref{eq.mom.befe} to deduce
\begin{multline}
\int_0^T\int_\Omega \varphi  (\overline{(p(\varrho) + \delta \varrho^8)} - (\lambda +2\mu) \diver \vv)\varrho \ {\rm d}x{\rm d}t\\
\label{rovnice.411}%
= \int_0^T\int_\Omega (\lambda + \mu) \diver \vv \nabla \varphi A(\varrho)\ {\rm d}x{\rm d}t -\int_0^T\int_\Omega \overline{(p(\varrho) + \delta \varrho^8)} \nabla \varphi A(\varrho)  \ {\rm d}x{\rm d}t + \mu\int_0^T\int_\Omega \nabla \vu \nabla \varphi A(\varrho)\ {\rm d}x{\rm d}t\\
-\mu \int_0^T\int_\Omega \vv \nabla \varphi \nabla A(\varrho)\ {\rm d}{\rm d}t + \mu \int_0^T\int_\Omega \vv \nabla \varphi \varrho \ {\rm d}x{\rm d}t - \int_0^T\int_\Omega \varrho \vv \otimes \vu \nabla \varphi A(\varrho)\ {\rm d}x{\rm d}t\\
-\int_0^T\int_\Omega\varrho \vu \pat \varphi A(\varrho)\ {\rm d}x{\rm d}t + \int_0^T\int_\Omega \varphi \varrho \vu (\partial_i A_j(\varrho \vv_i) - \vv_i \partial_i A_j(\varrho))\ {\rm d}x{\rm d}t
\end{multline}

We compare \eqref{rovnice.410} and \eqref{rovnice.411} and we use the fact that the Mikhlin-H\"ormander  multipliers theorem \cite[Theorem 5.2.7]{grafakos} together with aleready known convergencies yield
\begin{equation*}
\begin{split}
A(\varrho_\varepsilon)&\to A(\varrho)\ \mbox{ in }C([0,T]\times \overline\Omega)\\
\nabla A(\varrho_\varepsilon)&\to \nabla A(\varrho)\ \mbox{ in }L^\infty((0,T), L^\beta_{weak}(\Omega)).
\end{split}
\end{equation*}
Consequently
\begin{multline*}
\lim_{\varepsilon\to 0} \int_0^T\int_\Omega \varphi ( (\varrho_\varepsilon (p(\varrho_\varepsilon) + \delta\varrho_\varepsilon^8 - (\lambda + 2\mu)\diver \vu_\varepsilon) - (\overline{(p(\varrho) + \delta \varrho^8)} - (\lambda + 2\mu)\diver \vu)\varrho)\ {\rm d}x{\rm d}t\\
 = \lim_{\varepsilon\to 0}\int_0^T\int_\Omega \varphi \varrho_\varepsilon \vu_\varepsilon (\partial_i A_j(\varrho_\varepsilon \vv_{i\varepsilon}) - \vv_{i\varepsilon} \partial_i A_j(\varrho_\varepsilon))\ {\rm d}x{\rm d}t 
 - \int_0^T\int_\Omega \varphi \varrho \vu (\partial_i A_j(\varrho \vv_i) - \vv_i \partial_i A_j(\varrho))\ {\rm d}x{\rm d}t
\end{multline*}
and now it is sufficient to show that the right hand side is $0$. 

Using Div-Curl lemma one obtains
$$
\begin{array}{l}
\vv_n\to \vv \ \mbox{weakly in }L^p(\mathbb R^3)\\
\vw_n\to \vw \ \mbox{weakly in }L^q(\mathbb R^3)
\end{array}
\Rightarrow 
\vv_n \partial_i A_j(\vw_n) - \vw_n \partial_i A_j(\vv_n)\to \vv \partial_i A_j(\vw) - \vw \partial_i A_j(\vv)\ \mbox{weakly in }L^r(\mathbb R^3)
$$
whenever $r = \frac{pq}{p+q}<1$. For details of proof we refer to \cite[Lemma 3.4]{FeNoPe}. Further, \eqref{eq.rho.str} and \eqref{eq.mom.str} yield
$$
\varrho_\varepsilon \partial_iA_j (\varrho_\varepsilon \vv_{i\varepsilon}) - \varrho_\varepsilon \vv_{i\varepsilon} \partial_i A_j(\varrho_\varepsilon)\to \varrho \partial_i A_j(\varrho \vv) - \varrho \vv \partial_i A_j(\varrho)\ \mbox{strongly in }L^\infty((0,T),L^{2\beta/(\beta+3)}_{weak}(\Omega))
$$
and the claim follows due to the compact embedding of $L^{2\beta/(\beta+3)}(\Omega)$ to $W^{-1,2}(\Omega)$.
\end{proof}

Recall that $\varrho_\varepsilon$, $\vv_\varepsilon$ are regular enough to fulfill the renormalized continuity equation
$$
\pat b(\varrho_\varepsilon) + \diver (b(\varrho_\varepsilon)\vu_\varepsilon) + (b'(\varrho_\varepsilon)\varrho_\varepsilon - b(\varrho_\varepsilon))\diver \vu_\varepsilon - \varepsilon \Delta b(\varrho_\varepsilon)\leq 0
$$
for every $b$ convex and globally Lipschitz and the renormalized continuity equation is true also for $\varrho$ and $\vu$, namely
$$
\pat b(\varrho) + \diver(b(\varrho)\vu) + (b'(\varrho)\varrho - b(\varrho))\diver\vu =0.
$$
Taking $b(z) = z\log  z$ yields
$$
\int_0^\tau\int_\Omega \varrho_\varepsilon\diver \vu_\varepsilon \ {\rm d}x{\rm d}t\leq \int_\Omega \varrho_0\log \varrho_0\ {\rm d}x - \int_\Omega \varrho_\varepsilon(\tau,\cdot)\log \varrho_\varepsilon(\tau,\cdot)\ {\rm d}x
$$
and
$$
\int_0^\tau\int_\Omega \varrho \diver \vu \ {\rm d}x{\rm d}t = \int_\Omega \varrho_0\log \varrho_0 \ {\rm d}x - \int_\Omega \varrho(\tau,\cdot)\log \varrho(\tau,\cdot)\ {\rm d}x
$$
for every $\tau\in [0,T]$. Let $\Phi_n$ be a sequence of smooth compactly supported functions such that $\Phi_n\to 1$ in $L^q(\Omega)$ for sufficiently large $q$. We deduce
\begin{multline*}
\int_\Omega \varrho(\tau,\cdot)\log \varrho(\tau,\cdot) - \varrho_\varepsilon(\tau,\cdot) \log \varrho_\varepsilon(\tau,\cdot)\ {\rm d}x\geq \int_0^\tau\int_\Omega \varrho_\varepsilon\diver\vu_\varepsilon - \varrho\diver \vu\ {\rm d}x{\rm d}t\\
 = \int_0^\tau\int_\Omega\Phi_n\left( \varrho_\varepsilon\diver\vu_\varepsilon - \frac 1{\lambda + 2\mu}(p(\varrho_\varepsilon) + \delta\varrho_\varepsilon^8)\varrho_\varepsilon + \frac 1{\lambda + 2\mu}(p(\varrho_\varepsilon) + \delta\varrho_\varepsilon^8)\varrho_\varepsilon - \varrho\diver\vu\right) \ {\rm d}x{\rm d}t\\
 + \int_0^\tau \int_\Omega (1-\Phi_n)(\varrho_\varepsilon \diver \vu_\varepsilon - \varrho \diver \vu)\ {\rm d}x{\rm d}t
\end{multline*}
and we use \eqref{eq.e.v.f} in order to get
\begin{multline}\label{eq.convex.rho}
\int_\Omega \varrho(\tau,\cdot)\log \varrho(\tau,\cdot) - \overline{\varrho(\tau,\cdot)\log \varrho(\tau,\cdot)}\ {\rm d}x\geq \frac 1{\lambda + 2\mu} \int_0^\tau\int_\Omega \left(\overline{(p(\varrho) - \delta\varrho^8)\varrho} - \overline{(p(\varrho) - \delta\varrho^8)} \varrho\right)\Phi_n\ {\rm d}x{\rm d}t + \eta(n)\\
\geq 0 + \eta(n).
\end{multline}
where the last inequality is true because of the monotonicity of the mapping $\varrho\mapsto p(\varrho) + \delta\varrho^8$. Since $\eta(n)$ can be made arbitrarily small and since the mapping $\varrho\mapsto \varrho\log\varrho$ is convex, we get from \eqref{eq.convex.rho}
$$
\varrho_\varepsilon\to \varrho \ \mbox{almost everywhere in }(0,T)\times\Omega
$$
which, together with \eqref{eq.rho.int} yields \eqref{eq.tlak}.

Lastly, since \eqref{ene.ineq.app} yields
\begin{multline*}
-\int_0^T \left(\int_\Omega\frac 12 \varrho_\varepsilon |\vu_\varepsilon|^2 + \frac a{\gamma-1}\varrho_\varepsilon^\gamma + \frac \delta 7 \varrho_\varepsilon^8\ {\rm d}x + \frac k2 |\vb_\varepsilon(t)|^2\right) \pat \varphi\ {\rm d}t\\
 + \int_0^T \left(\int_\Omega S(\nabla 
\vv_\varepsilon):\nabla\vv_\varepsilon + \varepsilon a \gamma \varrho_\varepsilon^{\gamma-1}|\nabla\varrho_\varepsilon|^2 + 8\varepsilon\delta\varrho_\varepsilon^7|\nabla\varrho_\varepsilon|^2\ {\rm d}x - k\dot\vb_\varepsilon(t)\vf(t)\right) \varphi(t)\ {\rm d}t\\
 + \left(\int_\Omega \frac 12 \varrho_\varepsilon |\vu_\varepsilon|^2 + \frac a{\gamma-1}\varrho_\varepsilon^\gamma + \frac\delta8\varrho_\varepsilon^8\ {\rm d}x\right)\varphi(0)\leq 0
\end{multline*}
for all $\varphi\in C^\infty_c([0,T))$, $\varphi\geq 0$. We recall that due to the trace theorem (see e.g. \cite[Theorem 5.36]{AdFo}) we have
$\dot \vb_\varepsilon\in L^2(0,T)$ and we may use the Arz\'ela-Ascoli theorem to deduce 
$$\vb_\varepsilon \rightrightarrows \vb$$
Using the already known convergences we are able to deduce that
\begin{equation}\label{ene.po.eps}
\pat \int_\Omega \frac 12 \varrho|\vu|^2\ {\rm d}x + \pat \int_\Omega \frac a{\gamma-1}\varrho^\gamma + \frac \delta 7 \varrho^8\ {\rm d}x + \int_\Omega S(\nabla\vv):\nabla \vv\ {\rm d}x + \frac k2 \pat |\vb(t)|^2\leq k\dot\vb(t)\vf(t)
\end{equation}

\subsection{Limit $\delta\to 0$}

Let $\varrho_\delta,\ \vu_\delta( = \vv_\delta + \vb_\delta)$ be a solution constructed in the previous chapter corresponding to some positive parameter $\delta>0$. Here we are going to pass with $\delta$ to $0$. From \eqref{ene.po.eps} we deduce
\begin{equation}
\begin{split}\label{delta.odhady}
\mbox{esssup}_{t\in (0,T)} \|\varrho_\delta |\vu_\delta|^2\|_{L^1} & \leq c\\
\mbox{esssup}_{t\in (0,T)} \|\varrho_\delta\|_{L^\gamma} & \leq c\\
\|\vb_\delta\|_{L^\infty((0,T))} & \leq c\\
\|\vv_\delta\|_{L^2W^{1,2}} & \leq c\\
\|\vu_\delta\|_{L^2W^{1,2}} & \leq c
\end{split}
\end{equation}
and we also claim the following convergences
\begin{equation*}
\begin{split}
\varrho_\delta &\to \varrho\ \mbox{weakly* in }L^\infty((0,T),L^\gamma(\Omega))\\
\vu_\delta & \to \vu \ \mbox{weakly in }L^2((0,T),W^{1,2}(\Omega))\\
\vb_\delta & \to \vb \ \mbox{strongly in } L^\infty((0,T))
\end{split}
\end{equation*}
where the last one follows from the trace theorem and Arz\'ela-Ascoli theorem.

The density satisfies the following estimate
\begin{equation}\label{high.reg.rho}
\int_0^T\int_\Omega \varrho_\delta^{\gamma + \theta} \ {\rm d}x{\rm d}t\leq c
\end{equation}
for  some $\theta>0$ and with $c$ independent of $\delta$. Indeed, we take $\varphi = \mathcal B(\varrho_\delta^\theta - (\varrho_\delta^\theta)_{\Omega})$ for some $\theta>0$ specified later as a test function in \eqref{eq.mom.befe}. We obtain
\begin{multline*}
\int_0^T\int_\Omega (p(\varrho_\delta) + \delta \varrho_\delta^8)\varrho_\delta^\theta \ {\rm d}x{\rm d}t  = \int_0^T\int_\Omega (p(\varrho_\delta) + \delta\varrho_\delta^8)(\varrho_\delta^\theta)_\Omega \ {\rm d}x{\rm d}t\\
 + \int_0^T\int_\Omega \varrho_\delta \vu_\delta \mathcal B(\pat (\varrho_\delta^\theta - (\varrho_\delta^\theta)_\Omega))\ {\rm d}x{\rm d}t - \left[\int_\Omega \varrho_\delta\vu_\delta \mathcal B(\varrho_\delta^\theta - (\varrho_\delta^\theta)_\Omega)\ {\rm d}x\right]_{\tau = 0}^{\tau = T}\\
 + \int_0^T \int_\Omega \varrho_\delta \vv_\delta \otimes \vu_\delta \nabla \mathcal B(\varrho_\delta^\theta - (\varrho_\delta^\theta)_\Omega)\ {\rm d}x{\rm d}t + \int_0^T\int_\Omega S(\nabla \vv_\delta): \nabla \mathcal B(\varrho_\delta^\theta - (\varrho_\delta^\theta)_{\Omega})\ {\rm d}x{\rm d}t.
\end{multline*}
Recall that $\varrho_\delta \in L^{\gamma/\theta}(\Omega)$ and that $\mathcal B: L^p(\Omega)\mapsto W^{1,p}(\Omega)$. With this and \eqref{delta.odhady} at hand one can easily deduce estimates of all term on the right hand side with one exception -- the second term. There one has to use the continuity equation in order to get
\begin{multline*}
\int_0^T\int_\Omega \varrho_\delta \vu_\delta \mathcal B(\pat(\varrho_\delta^\theta - (\varrho_\delta^\theta)))\ {\rm d}x{\rm d}t\\
= \int_0^T\int_\Omega \varrho_\delta \vu_\delta \mathcal B \left(-\diver(\varrho_\delta^\theta \vv_\delta) - (\theta-1)\varrho^\theta \diver \vv_\delta + (\diver (\varrho_\delta^\theta \vv_\delta) + (\theta-1)\varrho_\delta^\theta \diver \vv_\delta)_\Omega\right)\ {\rm d}x{\rm d}t \\
= -\int_0^T\int_\Omega \varrho_\delta \vu_\delta \varrho_\delta^\theta \vv_\delta\ {\rm d}x{\rm d}t  - \int_0^T\int_\Omega \varrho_\delta \vu_\delta \mathcal B((\theta - 1)\varrho_\delta^\theta \diver \vv_\delta - ((\theta-1) \varrho_\delta^\theta \diver \vv_\delta)_\Omega)\ {\rm d}x{\rm d}t. 
\end{multline*}
Estimates \eqref{delta.odhady} yields $\varrho_\delta \vu_\delta \vv_\delta \in L^2((0,T),L^{6\gamma/(4\gamma + 3)}(\Omega))$ and thus  it is enough to choose $\theta$ sufficiently small to get $\varrho_\delta^\theta\in L^\infty ((0,T), L^{6\gamma/(2\gamma - 3)}(\Omega))$ and we get the control of the first term on the right hand side. The control of the second term follows easily as $\varrho_\delta^\theta\diver \vv_\delta \in L^2((0,T),L^{(2\gamma)/(\gamma+\theta)}(\Omega))$ and thus (with the help of Sobolev embedding theorem)
$$
\mathcal B ((\theta - 1)\varrho_\delta^\theta \diver \vv_\delta - ((\theta -1)\varrho_\delta^\theta \diver \vv_\delta)_\Omega) \in L^2((0,T), L^{6\gamma/(\gamma + 3\theta)}(\Omega)).
$$
Now, assuming $\theta$ is sufficiently small, we get the boundedness of the second term on the right hand side and we conclude \eqref{high.reg.rho}.

Similarly as before we claim that
\begin{equation*}
\begin{split}
\varrho_\delta \to \varrho\ \mbox{strongly in }L^p((0,T),W^{-1,2}(\Omega))\\
\varrho_\delta \vu_\delta \to \varrho \vu \ \mbox{strongly in }L^p((0,T),W^{-1,2}(\Omega))
\end{split}
\end{equation*}
for all $p\in (1,\infty)$. Here we used that $L^{2\gamma/(\gamma+1)}(\Omega)$ is compactly embedded into $W^{-1,2}(\Omega)$ assuming $\gamma>\frac32$. 

As a matter of fact, $\varrho,\ \vu$ satisfy
\begin{multline*}
\int_0^T\int_\Omega \varrho \vu \pat \varphi  + \varrho \vv \otimes \vu \nabla \varphi + S(\nabla \vv):D \varphi - \overline{p(\varrho)}\diver \varphi \ {\rm d}x{\rm d}t + \int_0^T \varphi{\restriction_{\partial \Omega}}(\vb(t) - \vf(t))k\ {\rm d}t\\
+ \int_\Omega (\varrho \vu)_0 \varphi(0,\cdot)\ {\rm d}x = 0. 
\end{multline*}
It remains to proof that $\varrho_\delta \to \varrho$ almost everywhere as then $\overline{p(\varrho)} = p(\varrho)$ and also the energy inequality can be deduced similarly as in the previous chapter.

Let $k\in \mathbb N$. We define
$$
T_k(z) = k T\left(\frac zk\right), \ z\in \mathbb R
$$
where $T$ is a smooth concave function satisfying
$$
T(z) = \left\{
\begin{array}{l}
z\ \mbox{for }z\leq 1\\
2\ \mbox{for }z\geq 3.
\end{array}
\right.
$$
In what follows, we will use that the following equality
\begin{equation}\label{effective.flux}
\lim_{\delta \to 0} \int_0^T\int_\Omega \varphi \left(p(\varrho_\delta) - (\lambda + 2\mu)\diver \vv_\delta\right)T_k(\varrho_\delta)\ {\rm d}x{\rm d}t 
= \int_0^T\int_\Omega \varphi \left(\overline{p(\varrho)} - (\lambda +2\mu\diver \vv \right) T_k (\varrho)\ {\rm d}x{\rm d}t
\end{equation}
holds for all smooth $\varphi$ with compact support in $(0,T)\times \Omega$. The proof of this equality is similar to the proof of \eqref{eq.e.v.f} with  only difference -- we take $\Phi_j = \varphi A_j(T_k(\varrho_\delta))$. Nevertheless, as all arguments are similar, we neglect the proof. 

Further, $\varrho,\vv$ solves the renormalized equation of continuity, namely
\begin{equation}\label{renorm.con}
\pat b(\varrho) + \diver (b(\varrho)\vv) + (b'(\varrho)\varrho - b(\varrho))\diver\vv = 0.
\end{equation}
in the weak sense for any $b\in C^1(\mathbb R)$ with $b'(z)=0$ for $z$ sufficiently large.

Indeed. First note that we can pass to a limit with $\varphi\to 1$ in \eqref{effective.flux} we deduce
\begin{equation}\label{flux.everywhere}
\lim_\delta\int_0^T\int_\Omega (p(\varrho_\delta) - (\lambda + 2\mu)\diver \vv_\delta ) T_k(\varrho_\delta)\ {\rm d}x{\rm d}t = \int_0^T\int_\Omega \left( \overline{p(\varrho)} - (\lambda + 2\mu \diver \vv) \right) T_k(\varrho)\ {\rm d}x{\rm d}t
\end{equation}
since integrands on both sides are equi-integrable provided $k$ is fixed. Next, we have
$$
\int_0^T\int_\Omega \varrho_\delta^\gamma T_k(\varrho_\delta) - \overline{\varrho^\gamma} \overline{T_k(\varrho)} \ {\rm d}x{\rm d}t \geq\ \int_0^T\int_\Omega (\varrho_\delta^\gamma - \varrho^\gamma) (T_k(\varrho_\delta) - T_k(\varrho))\geq \int_0^T\int_\Omega |T_k(\varrho_\delta) - T_k(\varrho)|^{\gamma+1}\ {\rm d}x{\rm d}t
$$
and
\begin{multline*}
\lim_{\delta\to 0} \int_0^T\int_\Omega \diver \vv_\delta T_k(\varrho_\delta ) - \diver \vv \overline{T_k(\varrho)}\ {\rm d}x{\rm d}t = \lim_{\delta\to 0}\int_0^T\int_\Omega \left( T_k(\varrho_\delta) - T_k(\varrho) + T_k(\varrho) - \overline{T_k(\varrho)}\right) \diver\vv_\delta \ {\rm d}x{\rm d}t\\
\leq 2 \|\diver \vv_\delta\|_{L^2((0,T)\times\Omega)} \|T_k(\varrho_\delta) - T_k(\varrho)\|_{L^2((0,T)\times \Omega)}.
\end{multline*}
With help of \eqref{flux.everywhere} we get
\begin{multline*}
0 = \lim_{\delta\to 0} \int_0^T\int_\Omega \varrho_\delta^\gamma T_k(\varrho_\delta) - \overline{\varrho^\gamma} \overline{T_k(\varrho)} - (\lambda + 2\mu)\left(\diver\vv_\delta T_k(\varrho_\delta) - \diver \vv  \overline{T_k(\varrho)}\right)\ {\rm d}x{\rm d}t\\
\geq \|T_k(\varrho_\delta) - T_k(\varrho) \|_{L^{\gamma+1}((0,T)\times\Omega)}^{\gamma+1} - 2\|\diver \vv_\delta\|_{L^2((0,T)\times\Omega)} \|T_k(\varrho_\delta) - T_k(\varrho)\|_{L^2((0,T)\times\Omega)}.
\end{multline*}
This yields
\begin{equation*}
\|T_k(\varrho_\delta) - T_k(\varrho)\|_{L^{\gamma+1}((0,T)\times\Omega)}\leq c
\end{equation*}
 Using a smoothening kernel we may deduce the validity of equation
\begin{equation}\label{pre.renorm}
\pat b(\overline{T_k(\varrho)}) - \diver (b(\overline{T_k(\varrho)})\vv) + (b'(\overline{T_k(\varrho)}) \overline{T_k(\varrho)} - b(\overline{T_k(\varrho)}))\diver \vv = b(\overline{T_k(\varrho)}) \left(\overline{(T_k(\varrho) - T'_k(\varrho)\varrho   )\diver \vv}\right)
\end{equation}
in $D'((0,T)\times \Omega)$. We plan to tend with $k$ to $\infty$. First, we have
$$
\overline{T_k(\varrho)}\to \varrho \ \mbox{ strongly in }L^p((0,T)\times\Omega),\ 1\leq p<\gamma.
$$
This follows as
\begin{equation*}
\|\overline{T_k(\varrho)} - \varrho\|^p_{L^p((0,T)\times\Omega)}\leq \mbox{liminf}_{\delta\to 0} \|T_k(\varrho_\delta) - \varrho_\delta\|_{L^p((0,T)\times\Omega)}^p\\
\leq \mbox{liminf}_{\delta\to 0} 2^pk^{p-\gamma}\|\varrho_\delta\|_{L^\gamma((0,T)\times\Omega}^\gamma\to 0.
\end{equation*}
for $k\to\infty$.\\
Since $b'(z) \equiv  0$ for $z\geq M$, we have
\begin{multline}\label{b.s.carkou}
\int_0^T\int_\Omega \left|b(\overline{T_k(\varrho)}) \left(\overline{(T_k(\varrho) - T'_k(\varrho)\varrho   )\diver \vv}\right)\right|\ {\rm d}x{\rm d}t \\
\leq  \sup_{0\leq z\leq M} |b'(z)| \|\vv_\delta\|_{L^2W^{1,2}} \mbox{liminf}_{\delta\to 0} \|T'_k(\varrho_\delta)\varrho_\delta - T_k(\varrho_\delta)\|_{L^2(Q_{k,m})}.
\end{multline}
where $Q_{k,m} = \{(t,x)\in (0,T)\times \Omega,\ \overline{T_k(\varrho)}\leq M\}$. The H\"older inequality yields
\begin{equation}\label{holder.ine}
\|T'_k(\varrho_\delta)\varrho_\delta - T_k(\varrho_\delta)\|_{L^2(Q_{k,m})}^2
\leq \|T'_k(\varrho_\delta)\varrho_\delta - T_k(\varrho_\delta)\|_{L^1((0,T)\times\Omega)}^{(\gamma-1)/\gamma} \|T'_k(\varrho_\delta)\varrho_\delta - T_k(\varrho_\delta)\|_{L^{\gamma+1}(Q_{k,m})}^{(\gamma+1)/\gamma}.
\end{equation}
Further, we have 
$$
\|T'_k(\varrho_\delta)\varrho_\delta - T_k(\varrho_\delta)\|_{L^1((0,T)\times\Omega)} \leq 2^\gamma k^{1-\gamma} \sup_{\delta} \|\varrho_\delta\|_{L^\gamma((0,T)\times\Omega)}^\gamma
$$
and, since $T'_k(z) z \leq T_k(z)$,
\begin{multline}
\|T'_k(\varrho_\delta)\varrho_\delta - T_k(\varrho_\delta)\|_{L^{\gamma+1}(Q_{k,m})}\leq 2\left(\|T_k(\varrho_\delta) - T_k(\varrho)\|_{L^{\gamma+1}((0,T)\times\Omega)} + \|T_k(\varrho)\|_{L^{\gamma+1}(Q_{k,m})}\right)\\
\leq 2\left(\|T_k(\varrho_\delta) - T_k(\varrho)\|_{L^{\gamma+1}((0,T)\times\Omega)} + \|T_k(\varrho) - \overline{T_k(\varrho)}\|_{L^{\gamma+1}(Q_{k,m})} + \|\overline{T_k(\varrho)}\|_{L^{\gamma+1}(Q_{k,M})}\right) \leq c\label{oscilacni.mira}
\end{multline}
The renormalize continuity equation \eqref{renorm.con} follows from \eqref{pre.renorm}, \eqref{b.s.carkou}, \eqref{holder.ine} and \eqref{oscilacni.mira}.

To conclude the proof we introduce
$$
L_k =\left\{
\begin{array}{l}
z\log z\mbox{ for } 0\leq z\leq k,\\
z\log k + z\int_k^z\frac1{s^2}T_k(s)\ {\rm d}s \mbox{ for }z\geq k.
\end{array}
\right.
$$

We use this as $b$ in the renormalized equation \eqref{renorm.con} to conclude
$$
\pat L_k(\varrho_\delta) + \diver(L_k(\varrho_\delta)\vv_\delta) + T_k(\varrho_\delta)\diver\vv_\delta = 0
$$
and
$$
\pat L_k(\varrho) + \diver(L_k(\varrho)\vv) + T_k(\varrho)\diver\vv = 0.
$$

Similarly to the previous section we deduce (compare with \eqref{eq.convex.rho})
$$
\int_\Omega L_k(\varrho(\tau,\cdot)) - \overline{L_k(\varrho(\tau,\cdot))}\ {\rm d}x \geq 0
$$
for almost all $\tau\in [0,T]$. We send $k\to\infty$ to deduce that $\varrho\log\varrho = \overline{\varrho\log\varrho}$ which yields $\varrho_\delta\to\varrho$ almost everywhere. The proof of Theorem \ref{app.exist} is finished.
\bigskip

{\bf Acknowledgement:} This research was funded by Czech Science Foundation project number Grant GA19-04243S  in the framework of RVO: 67985840.

\bibliographystyle{plain}
\bibliography{literatura}

\end{document}